\documentclass[10pt]{amsart}
\usepackage{amsmath}
\usepackage{amsfonts}
\usepackage{amssymb}

\newcommand\eps{\varepsilon}
\newcommand\R{{\mathbf{R}}}

\newcommand\Z{{\mathbf{Z}}}

\newcommand\N{{\mathcal{N}}}

\newcommand\T{{\mathbf{T}}}

\newcommand\supp{{\operatorname{Supp}}}
\theoremstyle{plain}
  \newtheorem{theorem}[subsection]{Theorem}
  
  \newtheorem{proposition}[subsection]{Proposition}
  \newtheorem{lemma}[subsection]{Lemma}
  \newtheorem{corollary}[subsection]{Corollary}

\theoremstyle{remark}
  \newtheorem{remark}[subsection]{Remark}

\theoremstyle{definition}

\begin{document}
\title[]{On localization of the Schr\"odinger maximal operator}
\author{Shuanglin Shao}
\address{IMA, University of Minnesota, Minneapolis, MN 55455}
\email{slshao@ima.umn.edu}

\vspace{-1in}
\begin{abstract}
In \cite{Lee:2006:schrod-converg}, when the spatial variable $x$ is localized, Lee observed that the Schr\"odinger maximal operator $e^{it\Delta}f(x)$ enjoys certain localization property in $t$ for frequency localized functions. In this note, we give an alternative proof of this observation by using the method of stationary phase, and then include two applications: the first is on is on the equivalence of the local and the global Schr\"odinger maximal inequalities; secondly the local Schr\"odinger maximal inequality holds for $f\in H^{3/8+}$, which implies that $e^{it\Delta}f$ converges to $f$ almost everywhere if $f\in H^{3/8+}$. These results are not new. In this note we would like to explore them from a slightly different perspective, where the analysis of the stationary phase plays an important role.
\end{abstract}
\date{\today}
\maketitle

\section{Introduction}
In \cite{Carleson:1980:some-analytic-problems}, Carleson raised a question: for what $f$, $e^{it\Delta}f$ converges to $f$ almost everywhere as $t$ goes to zero? where $e^{it\Delta}f(x):=\int_{\R^d} e^{ix\xi+it|\xi|^2} \hat{f}(\xi)d\xi.$
He showed that the answer is true if $f\in H^{1/4}(\R)$, where $H^s(\R^d)$ denotes the usual inhomogeneous Sobolev spaces. This condition on $f$ was found to be necessary by Dahlberg and Kenig, \cite{Dahlberg-Kenig:1982:converg-schrodinger}. It is not hard to generalize the counterexample obtained by Dahlberg, Kenig in \cite{Dahlberg-Kenig:1982:converg-schrodinger} to high dimensions; that is to say, the answer is no if $f\in H^{s}(\R^d)$ if $s<1/4$. When $d\ge 2$, Sj\"olin \cite{Sjolin:1987:converg-schrod}, and Vega \cite{Vega:1988:converg-schrod} have shown that $s>1/2$ is a sufficient condition independently. However the full question remains open when $d\ge 2$, though some exciting partial progress has been made when $d=2$, which is the dimension we will study in this paper.

As is well known, the study of the pointwise convergence problem is related to the boundedness of Schr\"odinger maximal operator: it is equivalent to asking, what is the least $s_0>0$ such that
\begin{equation}\label{eq-1}
\|\sup_{0\le t\le 1} |e^{it\Delta} f|\|_{L_x^2(B(x_0, 1))} \le C \|f\|_{H^{s_0+}},
\end{equation}where $B(x_0, 1)$ is a ball in $\R^2$, which we will take to be the unit ball $B(0,1)$ by translation invariance, and $s_0+$ denotes that $s_0+\eps$ for any $\eps>0$ and the constant $C$ depends on $\eps$. Thus $s_0=1/4$ is expected for all dimensions by the discussion above.

When $d=2$, Bourgain \cite{Bourgain:1992:remark-schrod, Bourgain:1995:new-oscillatory-integr} was the first to make some progress to show that there exists some $s<1/2$ such that \eqref{eq-1} holds; it was based on two improvements on Tomas-Stein's inequality for the paraboloids $e^{it\Delta}f:\, L^2\to L^4_{t,x}$: some new linear restriction estimates beyond Tomas-Stein's were found, and its $X_{p}$-space refinement. This framework and idea were refined and developed further in later works by Moyua, Vargas and Vega \cite{Moyua-Varg-Vega:1996:schrod-maxi, Moyua-Vargas-Vega:1999}, who showed that \eqref{eq-1} holds if $s>\kappa$ for some  $\kappa$ satisfying $20/41<\kappa <41/84$. The exponent was further improved by Tao and Vargas \cite{Tao-Vargas:2000:cone-2} to $s>15/32$, and by Tao \cite{Tao:2003:paraboloid-restri} to $s>2/5$, where the gain was a corollary of some new bilinear restriction estimates for paraboloids. These new estimates in \cite{Tao:2003:paraboloid-restri} were obtained by using the method of wave packets decomposition and induction on scales first introduced by Wolff \cite{Wolff:2001:restric-cone} to prove the sharp bilinear restriction estimates for the cones; it was adapted by Tao \cite{Tao:2003:paraboloid-restri} to establish the sharp bilinear estimates for the paraboloids. This framework designed by Tao in \cite{Tao:2003:paraboloid-restri} seems to be robust, which was used by Lee \cite{Lee:2006:schrod-converg} to improve further to $s>3/8$. It is interesting that Lee's result in \cite{Lee:2006:schrod-converg} is not coming from any new linear or bilinear restriction estimates; instead Lee made a crucial observation that the Schr\"dinger maximal operator enjoys some localization in time for frequency localized functions if the space is localized (cf. Theorem \ref{thm-localization}.).

We have found Lee's localization of the Schr\"odinger maximal operator for frequency localized functions interesting. Let us first recall the statement.
\begin{theorem}\label{thm-localization}
Suppose $f\in L^2$ satisfying that $\hat{f}$ is supported by $A(1):=\{1/2\le |\xi|\le 1\}$. Then the following implication holds,
If \begin{equation}\label{eq-3}
\|\sup_{0\le t\le N} |e^{it\Delta} f|\|_{L_x^2(B(0,N))} \lesssim N^{\alpha}\|f\|_2,
\end{equation}
then
\begin{equation}\label{eq-4a}
\|\sup_{0\le t\le N^2} |e^{it\Delta} f|\|_{L_x^2(B(0,N))} \lesssim  N^{\alpha}\|f\|_2,
\end{equation}
and
\begin{equation}\label{eq-4b}
\|\sup_{0\le t\le N} |e^{it\Delta} f|\|_{L_x^2(B(0, \lambda N))} \lesssim N^{\alpha}\|f\|_2
\end{equation} for any $\lambda \gg 1$.
\end{theorem}
It is essentially same as \cite[Lemma 2.3]{Lee:2006:schrod-converg}, which Lee established by using the wave packets decomposition of Schr\"odinger waves; for the definition of wave packets, see the statement in Lemma \ref{le-wave-packet}. In this note, we present an alternative argument by the stationary phase analysis in Section \ref{sec:localization}.

By Theorem \ref{thm-localization} and wave packets decomposition, Rogers  \cite{Rogers:2008:A-local-smoothing-estimate-for-Schrodinger} obtained the following equivalence between the ``local" bound \eqref{eq-local-c} and  the ``global" bound \eqref{eq-global-c} on the Schr\"odinger maximal functions. In a similar spirit we will use the analysis of stationary phase to give another proof in Section \ref{sec:equiv}.
\begin{theorem}\label{thm-equiv} Let $s>0$. Then the local bound
\begin{equation}\label{eq-local-c}
\left\|\sup_{0\le t\le 1} |e^{it\Delta}f|\right\|_{L^2(B(0,1))} \le C \|f\|_{H^{s+}(\R^2)},
\end{equation}is equivalent to the global bound,
\begin{equation}\label{eq-global-c}
\left\|\sup_{0\le t\le 1} |e^{it\Delta}f|\right\|_{L^2(\R^2)} \le C \|f\|_{H^{2s+}(\R^2)}.
\end{equation}
\end{theorem}
See \cite{Rogers:2008:A-local-smoothing-estimate-for-Schrodinger} for references on progress made on \eqref{eq-global-c}.

The second application of Theorem \ref{thm-localization} is to obtain the following local bound of the Schr\"odinger maximal operator \eqref{eq-1}, Theorem \ref{thm-loc-bound}, which is contained in \cite{Lee:2006:schrod-converg}. For completeness, we choose to follow the analysis in \cite{Tao:2003:paraboloid-restri} to present the argument. The new point in our argument is that we can take advantage of the localization in time to localize the spatial support of $f$ to a ball $B(0, 4N)$ and then apply the Bernstein inequality directly. Even though the time localization property is lost when we apply the Bernstein inequality to deal with the supermum in time, the spatial localization of $f$ enables us to gain it back. Thus we focus on establishing some bilinear estimate $L^1_xL^2_t(|x|\le N, t\sim N)$ instead of $L^1_xL^\infty_t(|x|\le N, t\sim N)$.
\begin{theorem}\label{thm-loc-bound}
\begin{equation}\label{eq-44}
\left\|\sup_{0\le t\le 1} |e^{it\Delta}f|\right\|_{L^2(B(0,1))} \le C \|f\|_{H^{s_0+}(\R^2)}, \text{ for }s_0=3/8.
\end{equation}
\end{theorem}
The proof of this theorem is in Section \ref{sec:bilinear}. As a corollary of the above theorems,
\begin{corollary}
\begin{enumerate}
\item
$\left\|\sup_{0\le t\le 1} |e^{it\Delta}f|\right\|_{L^2(\R^2)} \le C \|f\|_{H^{3/4+}(\R^2)}.$
\item If $f\in H^{3/8+}$, then $e^{it\Delta}f\to f$ almost everywhere as $t\to0$.
\end{enumerate}
\end{corollary}
\begin{remark}
As noted in \cite{Lee:2006:schrod-converg} and \cite{Rogers:2008:A-local-smoothing-estimate-for-Schrodinger}, one may obtain similar results for the phase function $e^{ix\xi+it\phi}$ where the phase $\phi$ satisfies: $|\partial_\xi^\alpha \phi|\le C|\xi|^{2-\alpha}$ and $|\nabla \phi|\ge C|\xi|$, and the Hessian of $\phi$ has two eigenvalues of the same sign. This condition on $\phi$ includes the ``elliptic type" defined in \cite{Moyua-Vargas-Vega:1999} and \cite{Tao-Vargas-Vega:1998:bilinear-restri-kakeya}.
\end{remark}
\begin{remark}
Concerning the analogous pointwise convergence question for the nonelliptic Schr\"odinger equation with the phase $\phi=\xi_1^2-\xi_2^2\pm\xi_3^2\pm\cdots \pm\xi_n^2$, the situation is different, where an almost complete answer is known in \cite{Rogers-Vargas-Vega:2006:pt-nonelliptic-NLS}. The sharp local bound was obtained: If $\Box =\partial_{x_1}^2-\partial_{x_2}^2\pm\partial_{x_3}^2\pm\cdots \pm\partial_{x_n}^2$, then
$$ \|\sup_{0\le t\le 1} |e^{it\Box}f|\|_{L^2{B(0,1)}}\le C_s \|f\|_{\dot{H}^{1/2}(\R^2)}.$$
In higher dimensions, there is only the endpoint $s_0=1/2$ missing.
\end{remark}
\textbf{Acknowledgements.} We would like to thank Sanghyuk Lee, Xiaochun Li, and Terence Tao for helpful comments.

\section{Notations}\label{sec:notation}
We list several notations which are used repeatedly in the paper. We fix the spatial dimension $d=2$ and $N\gg 1$.  For $f\in L^2$ satisfying that $\hat{f}$ is supported on the annulus $A(1):=\{1/2\le |\xi|\le 1\}$, we choose $\rho$ to be a suitable indicator function supported in $\{|\xi|\sim 1\}$ such that it equals $1$ on $\supp{f}$; we set
$K(x,t):=\int e^{ix\xi+it|\xi|^2}\rho(\xi)d\xi.$

Let $\chi_N$ be a bump function satisfying $\chi_N(y)=1$ for $|y|\le 4N$ and $=0$ for $|y|\ge 6N$.

$A\lessapprox B$ denotes the statement, for any $\eps>0$, $A\le C_\eps R^{C\eps} B$ for some large $R>0$ which will be clear in the context.

\section{Localization of the Schr\"odinger maximal operator}\label{sec:localization}
Let us first state a lemma before proving Theorem \ref{thm-localization}. Set $t_j:=jN, 1\le j\le N$ and $I_j:=[t_j, t_j-N]$; then $\cup I_j$ is a partition of $[0,N^2]$.
\begin{lemma}\label{le-lee-local}
Suppose $f\in L^2$ such that $\hat{f}$ is supported on $A(1)$. Then for $1\le j\le N$, there exists $f_j$ with compact Fourier support in $A(1)+O(N^{-1})$ such that for $|x|\le N$ and $t\in I_j$, we have
\begin{align}
\label{eq-37} & |e^{it\Delta}f(x)|\le |e^{i(t-t_j)\Delta}f_j(x)|+C_lN^{-l}\|f\|_2,\\
\label{eq-38} & \sum_j \|f_j\|^2_2 \le C\|f\|_2^2.
\end{align}
\end{lemma}
This statement is essentially same as \cite[Lemma 2.1]{Lee:2006:schrod-converg}. We present an alternative proof by using the stationary phase analysis.
\begin{proof} Let $\eta$ be a $C^\infty$ function with compact Fourier support in $B(0,1/2)$. Then by Poisson summation formula,
  $$ \sum_{k\in \Z^2} \eta (x-k)=1.$$
Let $\mathcal{X}=N\Z^2$. Then $\sum_{x_0\in \mathcal{X}} \eta(\frac {x-x_0}{N})=1.$  We write
\begin{equation}\label{eq-39}
\begin{split}
e^{it\Delta} f(x)
&=\int e^{ix\xi+i(t-t_j)|\xi|^2}\rho(\xi)e^{it_j|\xi|^2}\hat{f}(\xi)d\xi=\int \int e^{i(x-y)\xi+i(t-t_j)|\xi|^2}\rho(\xi)d\xi e^{-it_j\Delta}f dy\\
&=\sum_{y_0\in \mathcal{X}}\int \int e^{i(x-y)\xi+i(t-t_j)|\xi|^2}\rho(\xi)d\xi\eta(\frac {y-y_0}{N})e^{-it_j\Delta}f(y)dy.
\end{split}
\end{equation}Since $\partial_\xi((x-y)\xi+(t-t_j)|\xi|^2)=x-y+2(t-t_j)\xi$ and $1/2\le |\xi|\le1$, then for given $|x|\le N$, $|t-t_j|\le N$,
$$|x-y+2(t-t_j)\xi|\ge 10N \Rightarrow |y|\ge 4N. $$
In this case, the stationary phase estimate gives
$$\left|\int e^{i(x-y)\xi+i(t-t_j)|\xi|^2}\rho(\xi)d\xi \right|\lesssim N^{-100}(1+|y|)^{-100}.$$
This further gives
\begin{equation}\begin{split}
&\left| \sum_{y_0\in \mathcal{X}: |y_0|\ge 2N}\int \int e^{i(x-y)\xi+i(t-t_j)|\xi|^2}\rho(\xi)d\xi\eta(\frac {y-y_0}{N})e^{-it_j\Delta}f(y)dy \right|\\
&\lesssim N^{-100} \left|\sum_{y_0\in \mathcal{X}: |y_0|\ge 2N} \int (1+|y|)^{-100} \eta(\frac {y-y_0}{N})e^{-it_j\Delta}f(y)dy \right|\\
&\lesssim N^{-100} \left|\sum_{y_0\in \mathcal{X}: |y_0|\ge 2N} \int (1+|y|)^{-100} \eta(\frac {y-y_0}{N})e^{-it_j\Delta}f(y)dy \right|\\
&\lesssim N^{-100} \|(1+|y|)^{-100}\|_2 \left\|\bigl(\sum_{y_0\in \mathcal{X}: |y_0|\ge 2N} \eta(\frac {y-y_0}{N})\bigr)e^{-it_j\Delta}f(y) \right\|_2\\
&\lesssim N^{-100} \|f\|_2.
\end{split}
\end{equation}
So we may only concern the terms with $y_0\in \mathcal{X}$ and $|y_0|\le 2N$; these are finitely many terms. Without loss of generality, we assume $y_0=0$; so the term becomes
$$ \int \int e^{i(x-y)\xi+i(t-t_j)|\xi|^2}\rho(\xi)d\xi\eta(\frac {y}{N})e^{-it_j\Delta}f(y)dy.$$
Then $f_j$ is found by setting $f_j:=\eta(\frac {y}{N})e^{-it_j\Delta}f$. Thus \eqref{eq-37} follows.

To show \eqref{eq-38}, we assume $\hat{f}$ is supported on a cap $C(\xi_0, N^{-1})$of radius $N^{-1}$. The general case follows from a decomposition in the Fourier space and the orthogonality due to almost disjoint Fourier supports. Then we aim to show that $\sum_j \|f_j\|_2^2\lesssim \|f\|_2^2$, i.e.,
$$\sum_j \left\|\eta (\frac yN) e^{-it_j\Delta}f \right\|_2^2 \lesssim \|f\|_2^2.$$
We write out the left hand side,
\begin{equation}\label{eq-46}
\begin{split}
\sum_j \left\|\eta (\frac yN) e^{-it_j\Delta}f\right\|_2^2 &=\sum_j \int \eta^2(\frac yN) \left(\sum_{z_0\in \mathcal{X}}\int\int e^{i(y-z)\xi +it_j|\xi|^2} \rho(\xi) d\xi  \eta(\frac {z-z_0}{N})f(z)dz\right)^2 dy.
\end{split}
\end{equation}
We know that $\partial_\xi ((y-z)\xi +t_j|\xi|^2 )=(y-z) +2t_j \xi =-z+y+2t_j(\xi-\xi_0)+2t_j\xi_0$ and $t_j|\xi-\xi_0|\le 2N$; so for given $y$ and $t_j$, if $|-z+y+2t_j(\xi-\xi_0)+2t_j\xi_0|\ge 10N$, i.e., $|z-y-2t_j\xi_0|\ge 4N$, then
$$ \left|\int e^{i(y-z)\xi +it_j|\xi|^2} \rho(\xi) d\xi \right| \lesssim N^{-100} \left(1+|z-(y+2t_j\xi_0)|\right)^{-100}.$$
We collect such $z$ into a set $\Omega_{j,y}$. So by Cauchy-Schwarz,
\begin{equation}\label{eq-47}
\begin{split}
 &\left(\sum_{z_0\in \mathcal{X}}\int_{\Omega_{j,y}}\int e^{i(y-z)\xi +it_j|\xi|^2} \rho(\xi) d\xi  \eta(\frac {z-z_0}{N})f(z)dz\right)^2 \\
 &\lesssim N^{-200} \int_{\Omega_{j,y}} \left(1+|z-(y+2t_j\xi_0)|\right)^{-200} dz \int_{\Omega_{j,y}} \left| \sum_{z_0\in \mathcal{X}}\xi(\frac {z-z_0}{N})f(z)\right|^2 dz\\
 &\lesssim N^{-200} \|f\|_2^2
\end{split}
\end{equation} uniform in $j$ and $y$. so if plugging this back into \eqref{eq-46}, since $\#{j}\lesssim N$, we see that
\begin{equation}\label{eq-48}
\sum_j \int \eta^2(\frac yN) \left(\sum_{z_0\in \mathcal{X}}\int_{\Omega_{j,y}}\int e^{i(y-z)\xi +it_j|\xi|^2} \rho(\xi) d\xi  \eta(\frac {z-z_0}{N})f(z)dz\right)^2 dy\le N^{-50 } \|f\|^2_2.
\end{equation}
So we are left with the term
$$\sum_j \int \eta^2(\frac yN) \left(\sum_{z_0\in \mathcal{X}}\int \int e^{i(y-z)\xi +it_j|\xi|^2} \rho(\xi) d\xi  \eta(\frac {z-z_0}{N})f(z)dz\right)^2 dy,$$
where for given $j$ and $y$, $z$ is restricted to the set $\{z:\,|z-(y+2t_j\xi_0)|\le 4N\}$; under this constraint, there are only finitely many $z_0$'s in the summation $\sum_{z_0\in \mathcal{X}}$. Then we are left with proving
\begin{equation}\label{eq-49}
\sum_j \int \eta^2(\frac yN) \sum_{z_0\in \mathcal{X}} \left(\int \int e^{i(y-z)\xi +it_j|\xi|^2} \rho(\xi) d\xi  \rho(\frac {z-z_0}{N})\chi(\frac {z-z_0}{N})f(z)dz\right)^2 dy\lesssim \|f\|_2^2.
\end{equation}
Note that we have added a constraint $\chi(\frac {z-z_0}{N})$ in the inner integrand by the discussion above, where $\chi$ denotes a bump function adapted to the ball $B(0,C)$ for some $C>0$ uniform in $j$ and $y$. The left hand side of \eqref{eq-49} equals
\begin{equation}\label{eq-50}
\sum_{z_0\in \mathcal{X}}\int \eta^2(\frac {y}{N})\sum_j \left(\int \int e^{i(y-z)\xi +it_j|\xi|^2} \rho(\xi) d\xi  \rho(\frac {z-z_0}{N})\chi(\frac {z-z_0}{N})f(z)dz\right)^2 dy.
\end{equation}
Given $z_0$ and $y$, there are only finitely many $j$ such that $|(y-z)+2t_j\xi_0|\le N$, due to the fact that $|(t_j-t_i)\xi_0|\gtrsim N$ for $i\neq j$ and the compact support of $\chi$. We also know $\# j\lesssim N$, one more application of stationary phase analysis eliminates the summation in $j$ and then reduces to showing
\begin{equation}\label{eq-51}
\sum_{z_0\in \mathcal{X}} \int \eta^2(\frac yN)\left(\int \int e^{i(y-z)\xi +it_j|\xi|^2} \rho(\xi) d\xi  \rho(\frac {z-z_0}{N})\chi(\frac {z-z_0}{N})f(z)dz\right)^2 dy\lesssim \|f\|_2^2.
\end{equation}
By Plancherel, the left hand side of \eqref{eq-51}
\begin{equation}\label{eq-52}
\begin{split}
\lesssim &\sum_{z_0\in \mathcal{X}} \int \left(\int \int e^{i(y-z)\xi +it_j|\xi|^2} \rho(\xi) d\xi  \rho(\frac {z-z_0}{N})\chi(\frac {z-z_0}{N})f(z)dz\right)^2 dy\\
\lesssim &\sum_{z_0\in \mathcal{X}} \int \left(\int e^{iy\xi +it_j|\xi|^2} \rho(\xi) \mathcal{F}_z\bigl(\rho(\frac {z-z_0}{N})\chi(\frac {z-z_0}{N})f\bigr) d\xi\right)^2 dy\\
& =\sum_{z_0\in \mathcal{X}} \int \rho^2(\frac {z-z_0}{N})\chi^2 (\frac {z-z_0}{N})f^2(z)dz \lesssim \|f\|^2_2.
\end{split}
\end{equation}Here $\mathcal{F}_z(f)$ denotes the Fourier transform of $f$ in $z$. Hence \eqref{eq-38} follows. Therefore the proof of Lemma \ref{le-lee-local} is complete.
\end{proof}
Now we are ready to prove Theorem \ref{thm-localization}.
\begin{proof}[Proof of Theorem \ref{thm-localization}.]
We only deal with \eqref{eq-4a} as the proof of \eqref{eq-4b} is similar. By \eqref{eq-37}, for $|x|\le N$,
\begin{equation}\label{eq-42}
\sup_{0\le t\le N^2} |e^{it\Delta} f|^2 \le \sum_j \sup_{t\in I_j} |e^{it\Delta} f|^2 \le \sum_j \sup_{t\in I_j} |e^{i(t-t_j)\Delta} f_j|^2+ C_l N^{-l}\|f\|_2^2
\end{equation} for any $l\in \N$. Then we see that \eqref{eq-3} and \eqref{eq-38} imply
\begin{equation}
\begin{split}
\int_{|x|\le N} \sup_{0\le t\le N^2} |e^{it\Delta} f|^2 dx &
\le \int_{|x|\le N} \sup_{0\le s\le N} |e^{-is\Delta}f_j|^2 dx +C_lN^{-l} \|f\|_2^2 \\
&\le CN^{2\alpha}\sum_j \|f_j\|_2^2+ C_lN^{-l} \|f\|_2^2\\
&\le CN^{2\alpha}\|f\|_2^2.
\end{split}
\end{equation} Thus \eqref{eq-4a} follows.
\end{proof}

\section{Equivalence between local and global conjectures}\label{sec:equiv}
In this section, we use the stationary phase analysis to establish the equivalence between the local conjecture on Schr\"odinger maximal functions, Theorem \ref{thm-equiv}.

By Littlewood-Paley decompositions and Theorem \ref{thm-localization}, to establish the equivalence, it suffices to show that
\begin{proposition}\label{prop-equiv}
Let $s>0$ and $N\gg 1$. Suppose $f$ satisfies $\|f\|_2=1$, and $\supp\hat{f}$ is supported by $\{1/2\le|\xi|\le 1\}$. Then the following estimates are equivalent:
\begin{align}
\label{eq-30} & \left\|\sup_{0\le t\le N} |e^{it\Delta}f|\right\|_{L^2(B(0,N))} \le C N^{s+}, \text{ and } \\
\label{eq-31} &\left\|\sup_{0\le t\le N^2} |e^{it\Delta}f|\right\|_{L^2(\R^2)} \le C N^{2s+}.
\end{align}
\end{proposition}
\begin{proof}
Firstly \eqref{eq-31} implies \eqref{eq-30} easily. We need to establish the reverse implication. Suppose \eqref{eq-30} holds. For $0<t<N^2$, writing
\begin{equation}\label{eq-32}
e^{it\Delta}f(x)=\int\int e^{i(x-y)\xi+it|\xi|^2}\rho(\xi)d\xi f(y)dy =K(t) \ast f(x),
\end{equation}where $K(x,t):=\int e^{ix\xi+it|\xi|^2}\rho(\xi)d\xi$. We make the following observation, if we localize $x$ in a ball $B(x, N^2)$ for some $x\in \R^2$, in view of $|t|\le N^2$, then the stationary phase analysis on the kernel $K$ suggests that $y$ should be somehow localized in a slightly larger ball but with the same center, say $B(x,2N^2)$. Thus to show \eqref{eq-31}, we partition $\R^2$ into finitely overlapping balls $\cup_{j\in \Z^2} B(x_j, N^2)$. Then
\begin{equation}\label{eq-33}
\begin{split}
\left\|\sup_{0\le t\le N^2} |e^{it\Delta}f|\right\|_{L^2(\R^2)}^2 & \lesssim \sum_j \int_{B(x_j,N^2)}\left| \int\int e^{i(x-y)\xi+it|\xi|^2}\rho(\xi)d\xi f(y)dy\right|^2dx\\
&\lesssim \sum_j \int_{B(x_j,N^2)}\left| \int\int e^{i(x-y)\xi+it(x)|\xi|^2}\rho(\xi)d\xi f(y)\chi_{j,N}(y)dy\right|^2dx\\
&\qquad +\sum_j \int_{B(x_j,N^2)}\left| \int\int e^{i(x-y)\xi+it(x)|\xi|^2}\rho(\xi)d\xi f(1-\chi_{j,N})dy\right|^2dx,\\
&=\sum_j \int_{B(x_j,N^2)} \left|K(t(x))\ast f\chi_{j,N}\right|^2+\left|K(t(x))\ast f(1-\chi_{j,N})\right|^2 dx
\end{split}
\end{equation}where $|t(x)|\le N^2$ and $\chi_{j,N}$ is an indicator function of the ball $B(x_j, 4N^2)$ and supported in ball $B(x_j, 6N^2)$. By the stationary phase analysis on the kernel $K$, if $|y-x_j|>4N^2$ and $x\in B(x_j,N^2)$,
\begin{equation}\label{eq-34}
|K(x-y, t(x))|\lesssim C N^{-100}(1+|y-x|)^{-100}.
\end{equation}
For the second term above, by Schur's test, we have
\begin{equation}\label{eq-35}
\sum_j \int_{B(x_j,N^2)} \left|K(t(x))\ast f(1-\chi_{j,N})\right|^2 dx \lesssim \|f\|_2^2.
\end{equation}
From the first term above, by \eqref{eq-30}, we see that
\begin{equation}\label{eq-36}
\sum_j \int_{B(x_j,N^2)} \left|K(t(x))\ast f\chi_{j,N}\right|^2 dx \lesssim N^{2s+}\sum_j \int |f\chi_{j,N}|^2dy \lesssim N^{2s+} \|f\|_2^2.
\end{equation}
Thus \eqref{eq-31} follows from \eqref{eq-35} and \eqref{eq-36}. This completes the proof of Proposition \ref{prop-equiv}.
\end{proof}

\section{Proof of Theorem \ref{thm-loc-bound} via the wave packets decomposition and induction on scales}
In this section, we aim to show that
\begin{equation*}
\|\sup_{0\le t\le 1} |e^{it\Delta} f|\|_{L_x^2(B(0,1))} \le C\|f\|_{H^{s_0+}}, \text{ for } s_0=3/8.
\end{equation*}
Then by Littlewood-Paley decomposition and scaling, it reduces to show that, for $f\in L^2$ such that $\hat{f}$ is supported by $A(1)$, and for any $N\gg 1$,
\begin{equation}\label{eq-2}
\|\sup_{0\le t\le N^2} |e^{it\Delta} f|\|_{L_x^2(B(0,N))} \le C N^{s_0+}\|f\|_2, \text{ for } s_0=3/8.
\end{equation}
By Theorem \ref{thm-localization}, it reduces to prove, for the same $f$ as above,
\begin{equation}\label{eq-5}
\|\sup_{0\le t\le N} |e^{it\Delta} f|\|_{L_x^2(B(0,N))} \le CN^{s_0+}\|f\|_2, \text{ for } s_0=3/8.
\end{equation}
We are going to localize the support of $f$ in the spatial space by analyzing the kernel of $e^{it\Delta}f$ for a frequency localized $f$: Writing
$$e^{it\Delta}f(x)= \int \int e^{i(x-y)\xi +it|\xi|^2}\rho(\xi)d\xi f(y)dy.$$
From the stationary phase analysis, the critical point of the phase function $(x-y)\xi +t|\xi|^2$ occurs where $|x-y|\sim |t|$, which implies $|y|\le 2N$ as $|x|\le N$ and $|t|\le N$. So we split the integral above into two parts,
\begin{equation}\label{eq-6}
\int \int e^{i(x-y)\xi +it|\xi|^2}\rho(\xi)d\xi f(y)\chi_{N}(y)dy+\int \int e^{i(x-y)\xi +it|\xi|^2}\rho(\xi)d\xi f(1-\chi_{N})dy.
\end{equation}
By the kernel estimate, the second term is less than $CN^{-50}\|f\|_2$, which is thus negligible if integrating in $B(0,N)$. Hence we are focusing on proving
\begin{equation}\label{eq-7}
\left\|\sup_{0\le t\le N}\int \int e^{i(x-y)\xi +it|\xi|^2}\rho(\xi)d\xi f(y)\chi_{N}(y)dy\right\|_{L^2_x(B(0,N))}\le CN^{s_0+}\|f\|_2
\end{equation}
Let \begin{equation}\label{eq-43}
F(x,t):=\int \int e^{i(x-y)\xi +it|\xi|^2}\rho(\xi)d\xi f(y)\chi_{N}(y)dy.
 \end{equation} We denote the inverse Fourier transform of $f$ by $\mathcal{F}^{-1}(f)$ and $\delta$ denotes the Dirac mass. Then We rewrite \eqref{eq-43} as $$F(t,x)=\int e^{ix\xi+it|\xi|^2}\rho(\xi) \mathcal{F}_y(f\chi_{N})d\xi=\mathcal{F}^{-1}\left(\int \delta(\tau-|\xi|^2)e^{ix\xi}\rho\mathcal{F}(f\chi_N)d\xi \right),$$
which implies that the Fourier transform of $F$ in $t$ is supported on an interval of size $\sim 1$. Hence by Bernstein's inequality, for fixed $x$,
$$\sup_{0\le t\le N} |F(t)| \le \sup_{t\in \R} |F(t)| \le \|F\|_{L^4_t(\R)}.$$
Then to prove \eqref{eq-7}, it suffices to prove
\begin{equation}\label{eq-8}
\left\|F\right\|_{L^2_xL^4_t(B(0,N)\times \R)}= \left\|\int \int e^{i(x-y)\xi +it|\xi|^2}\rho(\xi)d\xi f\chi_{N}dy\right\|_{L^2_xL^4_t(B(0,N)\times \R)}\le CN^{s_0+}\|f\|_2.
\end{equation}
Because our function $f\chi_N$ has compact support $|x|\le N$, we can restrict $t$ to the interval $|t|\le 4N$ from the stationary phase analysis.

To prove \eqref{eq-8}, by time translation invariance, it suffices to establish
\begin{equation}\label{eq-9}
\left\|\int \int e^{i(x-y)\xi +it|\xi|^2}\rho(\xi)d\xi f\chi_{N}dy\right\|_{L^2_xL^4_t(Q_N)}\le CN^{s_0+}\|f\|_2,
\end{equation}where $$Q_N:=\{(x,t):\, x\in B(0,N), N/2\le t\le N\}. $$
This is implied by a more general estimate,
\begin{proposition}\label{prop-local-L2L4}
Let $\hat{f}$ be supported on $\{1/2\le |\xi|\le 1\}$. Then for $N\gg 1$,
\begin{equation}\label{eq-10}
\left\|\int e^{ix\xi +it|\xi|^2}\rho\hat{f}d\xi \right\|_{L^2_xL^4_t(Q_N)}\le CN^{s_0+}\|f\|_2 \text{ for }s_0=3/8.
\end{equation}
\end{proposition}
This proposition is proven by following a bilinear approach of using the wavepacket analysis and induction-on-scale argument due to Wolff \cite{Wolff:2001:restric-cone} and Tao \cite{Tao:2003:paraboloid-restri}; more precisely, we will follow the framework designed by Tao \cite{Tao:2003:paraboloid-restri}.

\begin{remark}
The desired estimate for the left hand side of \eqref{eq-10} is,  for $N\gg 1$,
\begin{equation}\label{eq-11}
\left\|\int e^{ix\xi +it|\xi|^2}\rho\hat{f}d\xi \right\|_{L^2_xL^4_t(Q_N)}\le CN^{1/4}\|f\|_2,
\end{equation}
which is sharp by the standard Knapp example, and answers the question raised at the beginning of the paper. However we are not able to establish it in this paper.
\end{remark}
\begin{remark}
By Cauchy-Schwarz's inequality and the Strichartz inequality, we see that
\begin{equation}\label{eq-12}
\left\|\int e^{ix\xi +it|\xi|^2}\rho\hat{f}d\xi \right\|_{L^2_xL^4_t(Q_N)}\le N^{1/2}\|\int e^{ix\xi +it|\xi|^2}\rho\hat{f}d\xi\|_{L^4_{x,t}}\lesssim N^{1/2}\|f\|_2,
\end{equation}
which recovers the estimates due to Sj\"olin \cite{Sjolin:1987:converg-schrod}, and Vega \cite{Vega:1988:converg-schrod}.
\end{remark}
\begin{proof}[Proof of Proposition \ref{prop-local-L2L4}]
By squaring \eqref{eq-10}, it suffices to prove
\begin{equation}\label{eq-13}
 \|e^{it\Delta} f e^{it\Delta}f\|_{L^1_xL^2_t(Q_N)}\lesssim N^{2s_0+} \|f\|_2 \|f\|_2 \text{ for }s_0=3/8.
\end{equation}
In order to utilize the bilinear approach in \cite{Tao:2003:paraboloid-restri}, we then apply the Whitney decomposition to create the ``transversality" condition between $e^{it\Delta} f$ and $e^{it\Delta}f$. For each $j\le 0$, we decompose the set $\{1/2\le |\xi|\le 1\}$ into a union of cubes $\tau^j_k$ which have disjoint interiors and sidelength $2^j$; we split each $\tau_k^j$ further into smaller ``disjoint" cubes of $\tau_{k}^{j-1}$ of side length $2^{j-1}$, and call this $\tau^j_k$ ``parents" of those $\tau^{j-1}_k$. We define $\tau^j_k \simeq \tau^j_{k'} $ if they are disjoint, but have adjacent ``parents". For $j\le 0$, $\{1/2\le |\xi|\le 1\}=\cup_k \tau^j_k$, so there exists a partition of unity  $\sum_{k} \phi^j_k=1$ with $\phi^j_k$ being subject to $\tau^j_k$. Then
\begin{equation}\label{eq-14}
e^{it\Delta}f e^{it\Delta } f=\sum_{j\le 0} \sum_{k} \sum_{k':\, \tau^j_k \simeq \tau^j_{k'}} e^{it\Delta}f^j_k e^{it\Delta}f^j_{k'},
\end{equation} where $\widehat{f_j^k}:=\hat{f} \phi_{\tau^j_k}$. It is easy to observe that, for each given $j$ and $k$, there are only finitely many $k'$ such that $\tau^j_k \simeq \tau^j_{k'}$. So by the triangle inequality, \eqref{eq-13} follows from
\begin{align}
&\label{eq-15} \text{For }2^{2j}N\le 1, \|e^{it\Delta} f^j_k e^{it\Delta}f^j_{k'}\|_{L^1_xL^2_t(Q_N)}\lesssim 2^{j/2}N^{2s_0} \|f^j_k\|_2 \|f^j_{k'}\|_2,\\
&\label{eq-16} \text{For }2^{2j}N\ge 1, \|e^{it\Delta} f^j_k e^{it\Delta}f^j_{k'}\|_{L^1_xL^2_t(Q_N)}\lesssim 2^{j/2+} N^{2s_0+} \|f^j_k\|_2 \|f^j_{k'}\|_2.
\end{align}
We will see that \eqref{eq-15} follows from Proposition \ref{prop-tiny-support}, while \eqref{eq-16} follows from Proposition \ref{prop-bilinear} and \eqref{eq-4b}.
\begin{proposition}\label{prop-tiny-support}For any $R>0$,
\begin{equation}\label{eq-15a}
\|e^{it\Delta}f\|_{L^2_{t,x}(Q_R)} \le CR^{1/2}\|f\|_2;
\end{equation}
If $\hat{f}$ is supported in a unit cube in $\{1/2\le |\xi|\le 1\}$ of sidelength $r$ with $r^2N\le 1$, then
\begin{equation}\label{eq-15b}
\|e^{it\Delta}f\|_{L^2_{x}L^\infty_t(Q_N)} \le CrN^{1/2}\|f\|_2.
\end{equation}
Both estimates are sharp.
\end{proposition}

We postpone the proof proposition to Section \ref{sec:tiny-supp}. Assuming this lemma, we see how it implies \eqref{eq-15}: If $2^{2j}N\le 1$,
\begin{equation}
\begin{split}
\text{LHS of } \eqref{eq-15} &\le \|e^{it\Delta} f^j_k\|_{L^2_{x,t}(Q_N)}\|e^{it\Delta}f^j_{k'}\|_{L^2_xL^\infty_t(Q_N)}\lesssim N^{1/2}2^jN^{1/2}\|f^j_k\|_2\|f^j_{k'}\|_2\\
&\le 2^{j/2} N^{3/4} (2^{2j}N)^{1/4} \|f^j_k\|_2\|f^j_{k'}\|_2 \le 2^{j/2} N^{3/4} \|f^j_k\|_2\|f^j_{k'}\|_2.
\end{split}
\end{equation} Thus \eqref{eq-15} follows.

Given $\lambda\ge 1$. Let $S_1:=\{\xi: |\xi_1-(\lambda+1/4)e_1|\le 1/50, |\xi_2|\le 1/50\}$, and $S_2:=\{\xi: |\xi_1-(\lambda e_1-1/4)|\le 1/50, |\xi_2|\le 1/50\}$.
In notion of \cite{Tao:2003:paraboloid-restri}, $S_1$ and $S_2$ are \emph{transverse}, i.e., the unit normals of $S_1$ and $S_2$ are separated by an angle $\gtrsim c>0$.
\begin{proposition}\label{prop-bilinear}
Suppose $\hat{f}$ and $\hat{g}$ are supported on $S_1$ and $S_2$ defined as above. Then for $R\ge 1$,
\begin{equation}\label{eq-17}
\|e^{it\Delta}f e^{it\Delta}g\|_{L^1_xL^2_t(|x|\le \lambda R, |t|\sim R)} \lesssim R^{2s_0+} \|f\|_2 \|g\|_2.
\end{equation}for $s_0=3/8$.
\end{proposition}
This is the key proposition, which is proven by following the analysis in \cite{Tao:2003:paraboloid-restri}; we postpone its proof to Section \ref{sec:bilinear}. Assuming it, we see how it implies \eqref{eq-16}. If $2^{2j}N\ge 1$,  by scaling,
\begin{equation}\label{eq-18}
\text{LHS of }\eqref{eq-16} =2^j \|e^{it\Delta} g^j_k e^{it\Delta}g^j_{k'}\|_{L^1_xL^2_t(Q_{j,N})}
\end{equation}
where $\widehat{g^j_k}(\xi):=\widehat{f^j_k}(2^j\xi)\phi^j_k(2^j\xi)$, likewise for $g^j_{k'}$; and $Q_{j,N}:=\{(x,t): \, |x|\le 2^j N, 2^{2j}N /2 \le t\le 2^{2j}N\}$. Then $g^j_k$ and $g^j_{k'}$ are supported around $|\xi_0|\sim 2^{-j}$; they are of size 1 and are transverse. Let $\lambda=2^{-j}$ in \eqref{eq-17}, then continuing \eqref{eq-18},
\begin{equation}
2^j\|e^{it\Delta}g^j_k e^{it\Delta}g^j_{k'}\|_{L^1_xL^2_t(|x|\le  2^j N, |t|\sim 2^{2j}N)} \lesssim 2^j (2^{2j}N)^{3/4+} \|g^j_k\|_2\|g^j_{k'}\|_2=2^{j/2+}N^{3/4} \|f^j_k\|_2\|f^j_{k'}\|_2,
\end{equation}
which implies \eqref{eq-16}. Thus the proof of Proposition \ref{prop-local-L2L4} is complete.
\end{proof}

Next we focus on proving Propositions \ref{prop-tiny-support} and \ref{prop-bilinear}, which is the content of the next two subsections.

\subsection{Proof of Proposition \ref{prop-tiny-support}.}\label{sec:tiny-supp}
In this section, we present a proof of Proposition \ref{prop-tiny-support}.
\begin{proof}
Firstly \eqref{eq-15a} follows directly from Plancherel's theorem:
$$\|e^{it\Delta}f\|_{L^2_{t,x}(Q_R)} \le \left(\int_{|t|\le R} \int_{\R^2} |e^{it\Delta}f|^2 dxdt\right)^{1/2}\lesssim R^{1/2}\|f\|_2.$$
This estimate is sharp by the standard Knapp example: Let $\hat{f}:=\chi_{\Omega}$, where $\chi_\Omega$ is a bump function supported by the set $\Omega:=\{(\xi_1,\xi_2):\, |\xi_1-e_1|\le R^{-1/2}, |\xi_2|\le R^{-1/2}\}$.

For \eqref{eq-15b}, let $\supp\hat{f}:=\Omega=\{(\xi_1,\xi_2):\, |\xi_1|\le r, |\xi_2-e_2|\le r\}$ and let $\chi_\Omega$ denote the indicator function of $\supp\hat{f}$. Then
$$\widehat{e^{it\Delta}f}=\widehat{e^{it\Delta}f}\chi_\Omega\Rightarrow e^{it\Delta}f=e^{it\Delta}f \ast_x \widehat{\chi_\Omega}$$ where $\ast_x$ denotes the convolution in $x$. Then
$$\sup|e^{it\Delta}f|\le \|e^{it\Delta}f\|_{L^2_t(\R)} \|\widehat{\chi_\Omega}\|_2=\|e^{it\Delta}f\|_{L^2_t(\R)} \|\chi_\Omega\|_2\lesssim r \|e^{it\Delta}f\|_{L^2_t(\R)}.$$
Then by Plancherel's theorem and Cauchy-Schwarz's inequality,
\begin{equation}
\begin{split}
\text{LHS of }\eqref{eq-15b} &\lesssim r \left(\int_{|x_2|\le N} \int_{x_1,t} |\int e^{ix_1\xi_1+it\tau} \bigl(\int_{\xi_2} e^{ix_2\xi_2}\delta(\tau-|\xi|^2) \hat{f}d\xi_2\bigr) d\xi_1d\tau |^2 dx_1dt\,dx_2\right)^{1/2}\\
&\lesssim r\left(\int_{|x_2|\le N} \int_{x_1,t} |\int |\int_{\xi_2} e^{ix_2\xi_2}\delta(\tau-|\xi|^2) \hat{f}d\xi_2|^2 d\xi_1d\tau \,dx_2\right)^{1/2}\\
&\lesssim rN^{1/2}\|f\|_2\sup_{\tau,|\xi_1|\le r} \bigl(\int_{\xi_2}\delta(\tau-|\xi|^2)d\xi_2\bigr)^{1/2}\\
&\lesssim rN^{1/2}\|f\|_2\sup_{\tau,|\xi_1|\le r} \bigl(\int_{\xi_2}\delta(\tau-\xi_1^2-\xi_2^2)d(\xi_2^2)\bigr)^{1/2}\\
&\lesssim rN^{1/2}\|f\|_2
\end{split}
\end{equation}as $|\xi_2|\gtrsim 1$, which implies \eqref{eq-15b}. This estimate is sharp: Let $\hat{f}:=\chi_{\Omega}$, where $\chi_\Omega$ is a bump function supported by the set $\Omega:=\{(\xi_1,\xi_2):\, |\xi_1|\le r, |\xi_2-e_2|\le r\}$; then if we restrict $(x,t)$ such that
$$|x_1|\le N^{1/2}/100\le r^{-1}/100, \,|x_2+2t|\le N^{1/2}/100\le r^{-1}/100, |t|\le N/100\le r^{-2}/100,$$
note that this restriction is possible since $r^2N\le 1$; then we see that $$\text{LHS of }\eqref{eq-15b} \gtrsim r^2(N^{1/2})^{2/2}=r^2N^{1/2}, \text{while}  \text{ RHS of }\eqref{eq-15b} \lesssim r^2N^{1/2}.$$
This shows that \eqref{eq-15b} is sharp. Thus we completes the proof of Proposition \ref{prop-tiny-support}.
\end{proof}

\subsection{Proof of Proposition \ref{prop-bilinear}.}\label{sec:bilinear}
\begin{proof}To show \eqref{eq-17}, we replace $f$ by $e^{i\lambda\xi_1}f$ and $g$ by $e^{i\lambda\xi_1}g$, then it suffices to prove
\begin{equation}\label{eq-19}
\|e^{it\Delta}f(x-2t\lambda e_1)e^{it\Delta}g(x-2t\lambda e_1)\|_{L^1_xL^2_t(|x|\le \lambda R, |t|\sim R)} \lesssim R^{3/4+} \|f\|_2 \|g\|_2
\end{equation}for $\hat{f}$ and $\hat{g}$ supported on $$\tilde{S}_1:=\{\xi:\, |\xi_1-3e_1/4|\le 1/50,\, |\xi_2| \le 1/50\}, \,\tilde{S}_2:=\{\xi:\, |\xi_1-5e_1/4|\le 1/50,\, |\xi_2| \le 1/50\}.$$ It is clear that $\tilde{S}_1$ and $\tilde{S}_2$ are transverse.

We partition the set $\{(x,t):\,|x|\le \lambda R, |t|\sim R)\}$ into $\sim \lambda^2$ spacetime parallelepipeds $$P_k:=B(x_k,R)\times (R/2,R), \,P_k \text{ makes an angle }\sim \lambda \text{ with } x \text{-plane}. $$
By the analysis to establish Theorem \ref{thm-localization}, for each $(x,t)\in P_k$, there exists $f_k$ and $g_k$ such that $\hat{f}_k, \,\hat{g}_k$ are supported in $\supp \hat{f}+CR^{-1}$ and $\supp \hat{g}+CR^{-1}$, respectively, and
\begin{equation}
\begin{split}
&|e^{it\Delta}f(x)| \le |e^{it\Delta}f_k(x)|+C_lR^{-l} \|f\|_2,\\
&|e^{it\Delta}g(x)| \le |e^{it\Delta}g_k(x)|+C_lR^{-l} \|g\|_2,\\
&\left(\sum_k \|f_k\|^2_2\right)^{1/2}\le C_\eps R^\eps \|f\|_2,\\
&\left(\sum_k \|g_k\|^2_2\right)^{1/2}\le C_\eps R^\eps \|g\|_2
\end{split}
\end{equation} for any $\eps>0$ and $l\in \N$, the set of natural numbers. Let $P$ denote any such parallelepiped. Then it suffices to prove, for the same $f$ and $g$ as above,
\begin{equation}\label{eq-20}
\|e^{it\Delta}fe^{it\Delta}g\|_{L^1_xL^2_t(P)} \lesssim R^{3/4+} \|f\|_2 \|g\|_2.
\end{equation}
We will follow the method of the wave packet decompositions and induction-on-scale in \cite{Tao:2003:paraboloid-restri} and \cite{Wolff:2001:restric-cone} to establish \eqref{eq-20}. We first observe that the above estimate is true for some large $\alpha>0$, i.e.,
\begin{equation}\label{eq-21}
\|e^{it\Delta}fe^{it\Delta}g\|_{L^1_xL^2_t(P)} \lesssim R^{\alpha} \|f\|_2 \|g\|_2.
\end{equation}
Then the claim \eqref{eq-20} will follow from the following inductive statement.
\begin{proposition}\label{prop-induction-hypo}
Suppose $\alpha>0$ is such that \eqref{eq-21} holds. Then
\begin{equation}\label{eq-22}
\|e^{it\Delta}fe^{it\Delta}g\|_{L^1_xL^2_t(P)} \le C_\eps R^{C\eps} \max\{ R^{(1-\delta)\alpha},\, R^{3/4+C\delta}\} \|f\|_2 \|g\|_2
\end{equation}for all $1\gg \eps, \delta>0$, where the constants $C>0$ independent of $\eps$ and $\delta$.
\end{proposition}By choosing suitable $\eps, \delta$ and performing iteration if necessary, \eqref{eq-20} and hence \eqref{eq-19} follows. Thus Proposition \ref{prop-induction-hypo} is all we need to establish.

We recall two crucial lemmas from \cite{Tao:2003:paraboloid-restri}: one on the wave packet decomposition of frequency-localized Schr\"odinger waves (\cite[Lemma 4.1]{Tao:2003:paraboloid-restri}), and the other on the \emph{relation} $\sim$ between tubes and balls (\cite[Eq. (19)]{Tao:2003:paraboloid-restri}), an $L^2$-localized spacetime estimate of wave packets (\cite[Eq. (23)]{Tao:2003:paraboloid-restri}). A typical $\tilde{S}_j$-tube is a set in the form of
$$T:=\{(t,x): R/2\le t\le R, \, |x-(x(T)+tv(T))|\le R^{1/2} \}$$
for an initial position $x(T)\in R^{1/2}\Z^2$ and initial velocity $v(T)\in R^{-1/2}\Z^2 \cap \tilde{S}_j$.  Let $\T_1,\,\T_2$ denote the collection of $\tilde{S}_1$-tubes and $\tilde{S}_2$-tubes respectively such that all the tubes intersect with $P$.
\begin{lemma}\label{le-wave-packet} Let $j=1,2$ and $f_j$ be a smooth functions on $\tilde{S}_j$. Then there exists a decomposition
$$e^{it\Delta}f_j=\sum_{T_j}c_{T_j}\phi_{T_j},$$
where $T_j$ ranges over all $\tilde{S}_j$-tubes, and the complex-valued coefficients $c_{T_j}$ satisfies $$\left(\sum_{T_j}|c_{T_j}|^2\right)^{1/2}\lesssim \|f_j\|_2,$$
and for each $T_j$, the wave packets $\phi_{T_j}$ are free Schr\"odinger waves, where for each $R/2\le t\le R$, the function $\phi_{T_j}(t)$ has Fourier transform supported on the set
$$ \xi\in \R^2: \xi=v(T_j)+O(R^{-1/2})$$
and obeys the pointwise estimates
$$\phi_{T_j}(t,x)\le C_lR^{-1/2} \left(1+\frac {|x-(x(T_j+tv(T_j)))|}{R^{1/2}}\right)^{-l}$$
for all $x\in R^2$ and any $l\in \N$, and
$$\|\sum_{T_j}\phi_{T_j}(t)\|_{L^2_x}\lesssim (\#\{T_j\})^{1/2}.$$
\end{lemma}
For any $\delta>0$, let $\mathfrak{p}:=\{p_k\}$ be a partition of $P$ consisting of smaller parallelepipeds of size $B(x_k, R^{1-\delta})\times (R^{1-\delta}/2, R^{1-\delta})$, $x_k\in \R^2$; it is clear that $\# k\le R^{C\delta}$. We will let $p$ denote any such $p_k$ and define $A \lessapprox B: A\le C_\eps R^\eps B$ for all $\eps>0$.
\begin{lemma}
Let the relation $\sim$ be defined as in \cite[Section 8]{Tao:2003:paraboloid-restri}. Then for all $T\in \T_1\cup \T_2$,
\begin{equation}\label{eq-23}
\#\{p\in \mathfrak{p}: T\sim p\}\lessapprox 1
\end{equation}
and
\begin{equation}\label{eq-24}
\begin{split}
\left\|\sum_{T_1\in \T_1: \,T_1\nsim p}\phi_{T_1} \sum_{T_2\in \T_2}\phi_{T_2}\right\|_{L^2(p)}&\lessapprox R^{C\delta} R^{-1/4} (\#\T_1)^{1/2}(\# \T_2)^{1/2},\\
\left\|\sum_{T_1\in \T_1}\phi_{T_1} \sum_{T_2\in \T_2:: \,T_2\nsim p}\phi_{T_2}\right\|_{L^2(p)}&\lessapprox R^{C\delta} R^{-1/4} (\#\T_1)^{1/2}(\# \T_2)^{1/2}.
\end{split}
\end{equation}
\end{lemma}
\begin{remark} The relation $\sim$ in \cite{Tao:2003:paraboloid-restri} is defined for tubes $T$ and balls $B$ of size $R^{1-\delta}$, and is for purpose of establishing \cite[Eq. (23)]{Tao:2003:paraboloid-restri}; one can see that \eqref{eq-23} is a byproduct of the definition of $\sim$ in \cite[Section 8]{Tao:2003:paraboloid-restri}. In our case, we need a relation between tubes and parallelepipeds of size $R^{1-\delta}$, which is determined by \eqref{eq-24}. We observe that \eqref{eq-24} is an $L^2$-estimate in both space and time; so we can exchange the integration order and make a change of variables, which reduces to the ``balls" case as in \cite[Eq. (23)]{Tao:2003:paraboloid-restri}. Thus the analysis in \cite{Tao:2003:paraboloid-restri} applies equally well.
\end{remark}
\begin{proof}[Proof of Proposition \ref{prop-induction-hypo}.]
We normalize $f,\,g$ such that $\|f\|_2=1$ and $\|g\|_2=1$. We apply Lemma \ref{le-wave-packet} to both $f$ and $g$, writing $f=\sum_{T_1} c_{T_1}\phi_{T_1}$ and $f=\sum_{T_2} c_{T_2}\phi_{T_2}$ for $T_i$ ranging over $\tilde{S}_i$-tubes, $i=1,2$. By the same pigeonholing as in deriving \cite[Eq. (15)]{Tao:2003:paraboloid-restri}, it suffices to prove
\begin{equation}\label{eq-25}
\left\|\sum_{T_1\in \T_1}\phi_{T_1} \sum_{T_2\in \T_2}\phi_{T_2}\right\|_{L^1_xL^2_t(P)}\lessapprox \bigl(R^{(1-\delta)\alpha}+ R^{3/4+C\delta}\bigr) (\#\T_1)^{1/2}(\# \T_2)^{1/2}
\end{equation} for all collections $\T_i$ of $\tilde{S}_i$-tubes such that all the tubes intersect $P$.

To prove \eqref{eq-25}, we partition $P$ into $O(R^{C\delta})$ finitely many overlapping spacetime parallelepipeds of size $R^{1-\delta}$ of with the same orientation: $P=\cup_k p_k$, where $p_k:=B(x_k, R^{1-\delta})\times (R^{1-\delta}/2, R^{1-\delta})$ for some $x_k\in \R^2$; recall that $\mathfrak{p}:=\{p_k\}$. We estimate the left hand side of \eqref{eq-25} by the triangle inequality,
$$ \sum_{p\in \mathfrak{p}} \left\|\sum_{T_1\in \T_1}\phi_{T_1} \sum_{T_2\in \T_2}\phi_{T_2}\right\|_{L^1_xL^2_t(p)}.$$
We split the above summation into four parts, $$\sum_{T_1\in \T_1: \,T_1\sim p}\sum_{T_2\in \T_2: \,T_2\sim p}, \text{ and } \sum_{T_1\nsim p}\sum_{T_2\sim p}, \text{ and } \sum_{T_1\sim p}\sum_{T_2\nsim p}, \text{ and } \sum_{T_1\nsim p}\sum_{T_2\nsim p};$$
By the triangle inequality again, we label corresponding sums by $I, II, III, IV$, for instance,
$$I:= \sum_{p\in \mathfrak{p}}\left\|\sum_{T_1\in \T_1: \,T_1\sim p}\phi_{T_1}\sum_{T_2\in \T_2: \,T_2\sim p}\phi_{T_2}\right\|_{L^1_xL^2_t(p)}.$$
By using \eqref{eq-23} and the induction hypothesis \eqref{eq-21}, and Cauchy-Schwarz's inequality,
\begin{equation}\label{eq-26}
I\lessapprox R^{(1-\delta)\alpha} (\#\T_1)^{1/2}(\# \T_2)^{1/2}.
\end{equation}
Thus we see that \eqref{eq-25} and hence Proposition \ref{prop-induction-hypo} follow from the following estimate
\begin{equation}\label{eq-27}
II+III+IV\lessapprox R^{3/4+C\delta} (\#\T_1)^{1/2}(\# \T_2)^{1/2}.
\end{equation}
Since $\#\{p_k\}\le R^{C\delta}$, \eqref{eq-27} is implies by
\begin{equation}\label{eq-28}
\left\|\sum_{T_1\nsim p}\phi_{T_1}\sum_{T_2}\phi_{T_2}\right\|_{L^1_xL^2_t(p)} \lessapprox R^{3/4+C\delta} (\#\T_1)^{1/2}(\# \T_2)^{1/2}
\end{equation} and
\begin{equation}\label{eq-29}
\left\|\sum_{T_1}\phi_{T_1}\sum_{T_2}\phi_{T_2\nsim p}\right\|_{L^1_xL^2_t(p)} \lessapprox R^{3/4+C\delta} (\#\T_1)^{1/2}(\# \T_2)^{1/2}
\end{equation}
for any $p\in \mathfrak{p}$. However, they are implied by \eqref{eq-24} combined with Cauchy-Schwarz's inequality in the spatial variable. Thus we finish the proof of \eqref{eq-25} and hence Proposition \ref{prop-induction-hypo}, which completes the proof of Proposition \ref{prop-bilinear}.
\end{proof}
\end{proof}


\end{document}